\documentclass[12pt]{article} 

\usepackage{amsmath}
\usepackage{amsthm}
\usepackage{amsfonts}
\usepackage{mathrsfs}
\usepackage{stmaryrd}
\usepackage{setspace}
\usepackage{fullpage}
\usepackage{amssymb}
\usepackage{breqn}
\usepackage{enumitem}
\usepackage{bbold} 
\usepackage{authblk}
\usepackage{comment}
\usepackage{hyperref}
\usepackage{pgf,tikz}
\usepackage{graphicx}
\usepackage{subcaption}

\newtheorem{thm}{Theorem}[section]%[chapter]

\newtheorem{lemma}[thm]{Lemma}
\newtheorem{conjecture}[thm]{Conjecture}

\newtheorem{prop}[thm]{Proposition}

\newtheorem{corollary}[thm]{Corollary}

\newtheorem{clm}[thm]{Claim}

\newcommand\ex{\ensuremath{\mathrm{ex}}}

\newcommand{\ignore}[1]{}

\title{A note on stability for maximal $F$-free graphs}
\author{D\'aniel Gerbner\footnote{Alfr\'ed R\'enyi Institute of Mathematics, E-mail: \texttt{gerbner@renyi.hu.} Research supported by the
    National Research, Development and Innovation Office -- NKFIH under the
    grants FK 132060, KKP-133819, KH130371 and SNN 129364.}}
\date{}

\begin{document}

\maketitle

\begin{abstract}
    Popielarz, Sahasrabudhe and Snyder in 2018 proved that maximal $K_{r+1}$-free graphs with $(1-\frac{1}{r})\frac{n^2}{2}-o(n^{\frac{r+1}{r}})$ edges contain a complete $r$-partite subgraph on $n-o(n)$ vertices. This was very recently extended to odd cycles in place of $K_3$ by Wang, Wang, Yang and Yuan. We further extend it to some other 3-chromatic graphs, and obtain some other stability results along the way.
\end{abstract}

\section{Introduction}

One of the most basic questions of graph theory is the following: given a graph $F$, how many edges can an $n$-vertex graph $G$ have if it is $F$-free, i.e. $G$ does not contain $F$ as a subgraph? This quantity is denoted by $\ex(n,F)$. Tur\'an's theorem \cite{T1941} states that among $n$-vertex $K_{r+1}$-free graphs, the most edges are in the complete $r$-partite graph with each partite set of order $\lfloor n/k\rfloor$ or $\lceil n/k\rceil$. This graph is now called the \textit{Tur\'an graph} and we denote it by $T_r(n)$. We denote the number of edges of $T_r(n)$ by $t_r(n)$.

The Erd\H os-Stone-Simonovits theorem \cite{ES1966,ES1946} states that it $r\ge 2$ and $F$ has chromatic number $r+1$, then $\ex(n,F)=(1+o(1))t_r(n)$. Erd\H os and Simonovits \cite{} showed that if an $n$-vertex graph $G$ is $F$-free and has almost $t_r(n)$ edges, then its structure is very similar to the structure of the Tur\'an graph. This phenomenon is called \textit{stability} and there are several non-equivalent stability theorems concerning the same graphs, where the differences come from the precise form of ``almost'', ``structure'' and ``very similar'' in the previous sentence.
In particular, the Erd\H os-Simonovits stability theorem \cite{erd1,erd2,sim} says that if $G$ is $F$-free on $n$ vertices with $t_r(n)-o(n^2)$ edges, then we can obtain $T_r(n)$ by adding and deleting $o(n^2)$ edges.

Tyomkin and Uzzel \cite{tyuz} initiated the study of new stability questions. We say that a graph is \textit{$F$-saturated} if it is $F$-free, but adding any new edge would create a copy of $F$. We also say that $G$ is \textit{maximal} with respect to the $F$-free property. When studying $\ex(n,F)$, one might assume without loss of generality that the $n$-vertex $F$-free graph $G$ is $F$-saturated, but if we consider the structure of $G$, this is a useful assumption.
Consider a $K_{r+1}$-saturated graph with close to $t_r(n)$ edges. Does it contain a large complete $r$-partite subgraph? Popielarz, Sahasrabuddhe and Snyder \cite{pss} answered this question with the following theorem.

\begin{thm}[Popielarz, Sahasrabuddhe and Snyder \cite{pss}]\label{posasn}
Let $r\ge 2$ be an integer. Every $K_{r+1}$-saturated graph $G$ on $n$ vertices with $t_r(n)-o(n^{\frac{r+1}{r}})$ edges contains a complete $r$-partite subgraph on $(1-o(1))n$ vertices. Moreover, there are $K_{r+1}$-saturated graphs on $n$ vertices with $t_r(n)-\Omega(n^{\frac{r+1}{r}})$ edges that do not contain a complete $r$-partite subgraph on $(1-o(1))n$ vertices.
\end{thm}

Wang, Wang, Yang and Yuan \cite{wwyy} considered the same problem for odd cycles in place of cliques and showed the following.

\begin{thm}\label{kinai}
Let $k\ge 2$ be an integer. Every $C_{2k+1}$-saturated graph $G$ on $n$ vertices with $t_2(n)-o(n^{\frac{3}{2}})$ edges contains a complete bipartite subgraph on $(1-o(1))n$ vertices. Moreover, there are $C_{2k+1}$-saturated graphs on $n$ vertices with $t_2(n)-\Omega(n^{\frac{3}{2}})$ edges that do not contain a complete bipartite subgraph on $(1-o(1))n$ vertices.
\end{thm}

Here we study the same problem for other graphs. Let us start with a bold conjecture.

\begin{conjecture}\label{nagysej}
Let $r\ge 2$ be an integer and $F$ be a graph with chromatic number $r+1$. Then every $F$-saturated graph $G$ on $n$ vertices with $t_r(n)-o(n^{\frac{r+1}{r}})$ edges contains a complete $r$-partite subgraph on $(1-o(1))n$ vertices.
\end{conjecture}

Observe that in case of forbidden cliques or cycles, the complete multipartite subgraph must be an induced subgraph. This is not the case in general, as we discuss in Section 2. 

If the above conjecture holds, the term $o(n^{\frac{r+1}{r}})$ cannot be improved in general by Theorem \ref{posasn}. Moreover, every 3-chromatic graph $F$ contains an odd cycle, thus in case $r=2$,  the term $o(n^{\frac{3}{2}})$ cannot be improved for any graph $F$ by Theorem \ref{kinai}.

If Conjecture \ref{nagysej} does not hold, weaker versions still should. Let us propose two such versions. We say that a vertex or edge of a graph is \textit{color-critical}, if deleting that vertex or edge results in a graph with smaller chromatic number.

\begin{conjecture}\label{kissej}
Let $r\ge 2$ be an integer and $F$ be a graph with chromatic number $r+1$ and a color-critical edge. Then every $F$-saturated graph $G$ on $n$ vertices with $t_r(n)-o(n^{\frac{r+1}{r}})$ edges contains a complete $r$-partite subgraph $G'$ on $(1-o(1))n$ vertices.
\end{conjecture}

We remark that in this case $G'$ has to be an induced subgraph. 
$K_r$ and $r$-chromatic graphs with a color-critical edge often behave similarly in extremal questions, see e.g. \cite{robscott} for several stability results. Another reason to assume that this conjecture might hold, and it might be easier to prove this than Conjecture \ref{nagysej} is the following.
The proofs of Theorems \ref{posasn} and \ref{kinai} both start with finding a not necessarily complete $r$-partite graph with many vertices and edges, using only the $F$-free property. After that, both  proofs continue with showing that $o(n)$ vertices are incident to all the missing edges between the two partite sets, thus removing those vertices finishes the proof.

The first step of this proof holds for every graph in this generality, but we need a very delicate version that fully uses the property of having $t_r(n)-o(n^{(r+1)/r}$ edges. Fortunately, the version due to Popielarz, Sahasrabudhe and Snyder (Lemma 2.3 in \cite{pss}) easily extends to graphs with a color-critical edge. The version in \cite{pss} uses a result of Andr\'asfai, Erd\H os and S\'os \cite{aes} that determined the largest possible minimum degree in an $n$-vertex $K_{r+1}$-free graph that is not $r$-partite, and a result of
Brouwer \cite{br} that determined the largest possible minimum number of edges in such graphs.

 Erd\H os and Simonovits \cite{ersim} extended the result of Andr\'asfai, Erd\H os and S\'os, while Simonovits \cite{miki} extended the result of Brouwer asymptotically to any $r$-chromatic graph with a color-critical edge in place of $K_r$. 
 Using those results instead, the lemma below easily follows by the same proof as Lemma 2.3 in \cite{pss}.

\begin{lemma}\label{harmad} Let $r\ge 2$ and $F$ be an $(r+1)$-chromatic graph with a critical edge.
Then there is a constant $d_F$, depending only on $F$, such that the following holds. If $0< \alpha$ is small enough, $n$ is large enough, and $G$ is an $n$-vertex $F$-free graph with $|E(G)|\ge t_{r}(n)-\alpha n^2$, then there is a subset $T\subset V(G)$ with $|T|\le d_F\alpha n$ such that $G-T$ is $(r-1)$-partite.
\end{lemma}

We omit the proof of this lemma. We will prove Conjecture \ref{kissej} for some 3-chromatic graphs with a color-critical edge, but we will use another lemma instead, making this paper self-contained.

\begin{thm}\label{elkrit}
Let $F$ be a 3-chromatic graph with a color-critical edge such that every edge has a vertex that is contained in a triangle. Then Conjecture \ref{nagysej} holds for $F$.
\end{thm}

Let us state a third conjecture, which is implied by the first and implies the second.

\begin{conjecture}\label{kozsej}
Let $r\ge 2$ be an integer and $F$ be a graph with chromatic number $r+1$ and a color-critical vertex. Then every $F$-saturated graph $G$ on $n$ vertices with $t_r(n)-o(n^{\frac{r+1}{r}})$ edges contains a complete $r$-partite subgraph on $(1-o(1))n$ vertices.
\end{conjecture}

A reason to assume that this conjecture might hold is that we are able prove it in the special case $r=2$ and the color-critical vertex is connected to every other vertex of $F$. 

\begin{thm}\label{main}
Let $F$ be a 3-partite graph with a vertex $w$ that is connected to every other vertex of $F$. Then Conjecture \ref{nagysej} holds for $F$.
\end{thm}

In our results, we are only interested in the order of magnitude and make no effort to optimize or even precisely state constant factors. We also assume basically everywhere that
$n$ is large enough, which means that there is a constant $n_0$ depending on the parameters introduced earlier such that $n\ge n_0$.

The rest of this paper is organized as follows. In Section 2, we prove some necessary lemmas and some unnecessary lemmas: related results that we do not use later. In Section 3, we present the proof of Theorems \ref{main} and \ref{elkrit}.

\section{Lemmas and other results}
The main reason to assume that Theorem \ref{posasn} can be extended as in Conjecture \ref{nagysej} is that stability results often extend in a similar way. Let us show an example.

\begin{thm}[Nikiforov, Rousseau]\label{nikro}
For $r\ge 3$ there is a constant $d_r$, depending only on $r$, such that the following
holds. For every $0<\alpha \le d_r$, every $K_r$-free $n$-vertex graph $G$
with at least $(\frac{r-2}{2r-2}-\alpha)n^2$ edges
contains an induced $r$-chromatic graph $G'$ of order at least $(1-2\alpha^{1/3})n$ and with
minimum degree
at least $(\frac{r-2}{r-1}-4\alpha^{1/3})n$.
\end{thm}

We can extend the above theorem to any $r$-chromatic graph.

\begin{prop}\label{elso}
For $r\ge 3$ there is a constant $d'_r$, depending only on $r$, such that the following
holds. 
Let $F$ be an $r$-chromatic graph and $n$ be large enough. For every $0<\alpha' \le d_r'$, 
every $F$-free $n$-vertex graph $G$
with at least $(\frac{r-2}{2r-2}-\alpha')n^2$ edges
contains an $r$-chromatic graph $G'$ of order at least $(1-2(\alpha')^{1/3})n$ and with
minimum degree
at least $(\frac{r-2}{r-1}-4(\alpha')^{1/3})n$.
\end{prop}

\begin{proof} 
By a result of Alon and Shikhelman \cite{as}, for any $\varepsilon>0$, if $n$ is large enough, then any $n$-vertex $F$-free graph contains at most $\varepsilon n^{|V(H)|}$ copies of $K_r$. By the removal lemma, for any $\delta>0$ there is $\delta>0$ such that if an $n$-vertex graph contains  at most $\varepsilon n^{|V(H)}$ copies of $K_r$, then we can delete at most $\delta n^2$ edges to obtain a $K_r$-free graph. Let $d_r'$ be any number smaller than $d_r$ from Theorem \ref{nikro}, $\delta\le d_r-\alpha'$ be a constant, $\varepsilon$ be as needed to apply the removal lemma, and $n$ be large enough so that we can use the result of Alon and Shikhelman. Then we can apply the removal lemma and delete $\delta n^2$ edges to obtain a $K_r$-free graph. Then we can apply Theorem \ref{nikro} to this graph to find the desired $G'$.
\end{proof}

The simple proof of the above proposition works for many other stability results concerning $K_{r+1}$. Let us mention another example without going into details: Kor\'andi, Roberts and Scott \cite{krs} considered $K_{r+1}$-free graphs with at least $t_r(n)-\delta_r n^2$ edges, and determined the largest number of edges one may need to remove from such a graph to obtain an $r$-partite graph. If we consider an $F$-free graph where $F$ has chromatic number $r+1$, then we delete $o(n^2)$ edges first to remove the copies of $K_{r+1}$ as in the above, and then apply their theorem to obtain an upper bound on the number of edges we additionally need to remove. This bound will not be sharp for two reason: we started with more edges (by $o(n^2)$, and we also removed those edges), and their $K_{r+1}$-free construction showing the sharpness of their result may contain $F$. If $F$ contains $K_{r+1}$, the second problem does not occur, and we obtain an asymptotically sharp result.

\smallskip

The above proof method, i.e. the combination of the result of Alon and Shikhelman and the removal lemma does not help with Conjecture \ref{nagysej}, as we have to remove almost quadratic many edges when using the removal lemma. We can prove a stronger lemma for a much smaller class of graphs.

\begin{lemma}\label{harmad2} Let $F$ be a $3$-chromatic graph with a color-critical vertex and $n$ be large enough. 
Let $\frac{20|V(F)|}{n}<\alpha<\frac{1}{11|V(F)|^2}$. If $G$ is an $n$-vertex $F$-free graph with $|E(G)|\ge \ex(n,F)-\alpha n^2$, then there is a bipartite subgraph $H$ of $G$ with at least $(1-12|V(F)|\alpha)n$ vertices, at least $\ex(n,F)-13k\alpha n^2$ edges and minimum degree at least $\left(\frac{1}{2}-\frac{1}{11|V(F)|}\right)n$ such that every vertex of $H$ is adjacent in $G$ to at most $|V(F)|$ vertices in the same partite set of $H$. 
\end{lemma}

\begin{proof} Observe that $F$ is a subgraph of $K_{1,k,k}$ for some $k$. Simonovits \cite{miki} showed that for the complete $(r+1)$-partite graph $K=K_{1,k,\dots,k}$ we have $\ex(n,K)\le t_r(n)+kn$ (in fact he obtained a more general result, that implies an exact result for $\ex(n,K)$, and he also described the asymptotic structure). This implies that $\ex(n,F)\le \frac{1}{4}n^2+kn$. 

We start by removing vertices of small degree, like the proofs of Theorem \ref{posasn} in \cite{pss} and Theorem \ref{kinai} in \cite{wwyy}. 
Let $G_0=G$ and given $G_i$ on $n_i=n-i$ vertices, if every vertex of $G_i$ has degree at least $(\frac{1}{2}-\frac{1}{11k})n_i$, then we let $G'=G_i$. If there is a vertex $v$ with degree less than $(\frac{1}{2}-\frac{1}{11k})n_i$, then we let $G_{i+1}$ be the graph obtained from $G_i$ by deleting $v$.
Let $n'$ be the number of vertices of $G'$. We have \begin{equation}\label{eq1}
|E(G)|\le |E(G')|+\sum_{i=0}^{n-n'-1} (\frac{1}{2}-\frac{1}{11k})(n-i).    
\end{equation} 

The right hand side of (\ref{eq1}) is at most $\frac{1}{4}n^2-\frac{1}{11k}(n^2-{n'}^{2})/2+kn'$, while the left hand is at least $\frac{1}{4}n^2-\alpha n^2$. This shows that $n'\ge (1-12k\alpha)n$. 

Let us consider a partition of $G'$ into two parts $A$ and $B$ with the most edges between parts. By the Erd\H os-Simonovits stability theorem, there are $o(n^2)$ edges inside the parts. We also have that if a vertex $v$ is connected to $d$ vertices in its part, say $A$, then it is connected to at least $d$ vertices in the other part $B$. On the other hand, $v$ is connected to at least $(\frac{1}{2}-\frac{1}{11k})n'-d$ vertices of $B$. This implies that $d\le (\frac{1}{4}-\frac{1}{22k})n'$. Thus we have that $v$ (and every other vertex) is connected to at least $(\frac{1}{4}-\frac{1}{22k})n'$ vertices in the other part. 

Let us assume that $|A|\ge |B|$. We call a vertex in $A$ \textit{good} if it is connected to at most $\frac{1}{22k}n'$ vertices in its part, and \textit{bad} otherwise. Observe that a good vertex has at least $(\frac{1}{2}-\frac{3}{22k})n'\ge |B|-\frac{3}{22k}n'$ neighbors in $B$. If $u\in A$ is connected to $k$ good vertices in $A$, then these $k+1$ vertices have at least $k$ common neighbors in $B$, thus there is a $K_{1,k,k}$ in $G$, a contradiction. Indeed, $u$ has at least $(\frac{1}{4}-\frac{1}{22k})n'$ neighbors in $B$, and each good vertex in $A$ is connected to all but at most $\frac{3}{22k}$ of these vertices.
thus the $k+1$ vertices chosen from $A$ have at least $(\frac{1}{4}-\frac{1+3k}{22k})n'\ge k$ common neighbors in $B$. This shows that every bad vertex has at least $\frac{1}{22k}n'-k$ bad neighbors in its part. If there exists a bad vertex, then there are $\Theta(n)$ bad vertices. Each of them is connected to $\Theta(n)$ vertices in the same part, thus there are $\Theta(n^2)$ edges inside the parts, a contradiction.

Therefore, we can assume that every vertex of $A$ is good, thus every vertex in $A$ is connected to at least $(\frac{1}{2}-\frac{3}{22k})n'$ vertices on the other side. This in particular shows that $(\frac{1}{2}-\frac{3}{22k})n'\le |A|,|B| \le (\frac{1}{2}+\frac{3}{22k})n'$.
Now we call a vertex in $A$ \textit{good} if it is connected to at most $\frac{1}{22k}n'$ vertices in its part, and \textit{bad} otherwise. Then a good vertex in $B$ has at least $(\frac{1}{2}-\frac{3}{22k})n'\ge |A|-\frac{6}{22k}n'$ neighbors in $A$. By the same reasoning as for bad vertices in $A$, we obtain that the existence of one bad vertex in $B$ would imply the existence of $\Theta(n)$ bad vertices in $B$ and $\Theta(n^2)$ edges inside the parts, a contradiction. Thus we can assume that every vertex is good.

Assume now that a vertex $v\in A$ is connected to at least $k$ vertices in $A$. Then $v$ and $k$ of its neighbors in $A$ are each connected to all but at most $\frac{3}{11k}n'$ vertices of $B$. Thus the number of vertices in $B$ that are not connected to some of them is at most $\frac{3}{11k}(k+1)n'$. Therefore, there are at least $k$ other vertices in $B$, those are common neighbors of the $k+1$ vertices picked earlier, hence they form a copy of $K_{1,k,k}$, a contradiction.

This shows that there are at most $(k-1)n<\alpha n^2$ edges inside $A$ and $B$.
Let us delete all the edges inside the parts $A$ and $B$ from $G'$ to obtain $H$. Clearly we have deleted at most $12k\alpha n^2$ edges to get $G'$ and at most $\alpha n^2$ edges to get $H$. The bounds on the number of vertices and the minimum degree of $H$ are obvious.
\end{proof}

Recall that the bipartite graph we found is not necessarily induced. However, for some graphs we can strengthen the above result.

\begin{corollary}\label{corola} Let $F$ be a 3-chromatic graph with a critical vertex such that $F$ can also be obtained from a bipartite graph by adding a matching into one of the parts.
Let $n$ be large enough and $\frac{20|V(F)|}{n}<\alpha<\frac{1}{11|V(F)|^2}$. If $G$ is an $n$-vertex $F$-free graph with $|E(G)|\ge \ex(n,F)-\alpha n^2$, then there is an induced bipartite subgraph $H'$ of $G$ with at least $(1-13\alpha)n$ vertices, at least $\ex(n,F_k)-14k\alpha n^2$ edges and minimum degree at least $(\frac{1}{2}-\frac{1}{10|V(F)|})n$. 
\end{corollary}

\begin{proof} Let $G'$ be the graph as in the proof of Lemma \ref{harmad2}. 
We will show that inside the partite sets $A,B$ of $G'$ there is no matching with $|V(F)|$ edges. Indeed, the $2|V(F)|$ vertices of that matching would have a common neighbor on the other partite set by the minimum degree condition, but this way we find a copy of $F$, a contradiction. This and the bound on the maximum degree implies that the subgraph of $G$ inside $A$ and inside $B$ has $(2|V(F)|-2)|V(F)|-1$ edges. Indeed, every edge shares a vertex with at most $2|V(F)|-2$ other edges, thus we can greedily pick $|V(F)|$ independent edges. See \cite{ahs} and \cite{chha} for more precise bounds on the number of edges.   

 Therefore, we can remove the endpoints of those $O(1)$ edges to obtain $H'$. It is easy to see that $H'$ has the desired number of vertices, edges and minimum degree, using that $n$ is large enough.
\end{proof}

We remark that if $F$ cannot be obtained from a bipartite graph by adding a matching into one of the parts, then a similar strengthening is impossible. Indeed, if we add a matching to one of the parts of the Tur\'an graph, the resulting graph is $F$-free, and we need to remove about $n/4$ vertices to obtain an induced complete bipartite graph.

\section{Proofs}

\begin{proof}[Proof of Theorem \ref{main}] Let $n$ be large enough and $G$ be an $n$-vertex $F$-saturated graph with $t_2(n)-o(n^{3/2})$ edges.
First we apply Lemma \ref{harmad2} with $\alpha=o(n^{-1/2})$ to obtain $H$ with partite sets $A$ and $B$. We will also use the subgraph $G'$ of $G$ from the proof of Lemma \ref{harmad2}, that is $H$ with additional edges inside the parts, such that every vertex is incident to at most $|V(F)|$ such edges. Let $T$ denote the set of vertices not in $G'$, thus $|T|=o(n^{1/2})$. For $v\in T$, let $A(v)$ denote its neighborhood in $A$ and $B(v)$ denote its neighborhood in $B$. Let $U(v)$ denote the smaller of $A(v)$ and $B(v)$ (if they have the same number of vertices, we choose one of them arbitrarily). Let $U_0=\cup_{v\in T} U(v)$ and for $1\le i\le |V(F)|$, $U_i$ denotes the set of vertices that are connected to a vertex of $U_{i-1}$ in the same parts. 
 As the degrees inside $A$ and $B$ are at most $|V(F)|$, we have that $|U_{|V(F)|}|\le |V(F)|^{|V(F)|}|U_0|$.
 
 Let us now consider a partition of $F$ to two connected subgraphs $F_0$ and $F_1$ such that $w$ (the vertex that is connected to every other vertex of $F$) is in $F_1$. Let $Q$ denote the subgraph induced on the vertices in $F_1$ that are connected to some vertices of $F_0$. Then $w$ is in $Q$, thus $Q$ is also connected. 
 
 Consider the copies of $F_0$ inside $A$ and those copies of $Q$ inside $B$ that can be extended to a copy of $F_1$ in $G$ (note that we do not care where the additional vertices come from or how many such extensions exist). Observe that every vertex $v\in A$ is contained in $O(1)$ copies of $F_0$. Indeed, as $F_0$ is connected, there is a path of length at most $|V(F)|$ from $v$ to every vertex of such copies, and there are at most $|V(F)|^{|V(F)|}$ vertices in $A$ that can be reached from $v$ by a path of length at most $|V(F)|$ that is totally inside $A$. If $v\in B$, then there are $O(1)$ copies of $Q$ containing it by the same reasoning (in fact there is a path of length at most 2 from $v$ to other vertices of $Q$ in this case).
 
 Let us assume that there are at least $n^{3/4}$ copies of $F_0$ inside $A$ and at least $n^{3/4}$ copies of $Q$ inside $B$. For any such copy of $Q$, we pick an extension to $F_1$, and observe that it intersects $O(1)$ copies of $F_0$ inside $A$. For the other copies of $F_0$, there is a vertex $u$ in the copy of $F_0$ and a vertex $v$ in the copy of $Q$ such that $uv$ is not an edge in $G$, by the $F$-free property. This way for $n^{3/2}-O(n^{3/4})$ pairs of $F_0$ and $F_1$, we found a missing edge. As both $u$ and $v$ are counted $O(1)$ times, this means that $\Omega(n^{3/2})$ edges between $A$ and $B$ are missing from $G$. 
 There are at most $n^2/4-\Omega(n^{3/2})$ edges of $G$ between $A$ and $B$, $O(n)$ edges inside $A$ and $B$, and at most $n|T|=o(n^{3/2})$ edges of $G$ are incident to $T$. Therefore, the total number of edges of $G$ is at most the sum of these, contradicting our assumption.
 
 We obtained that there are less than $n^{3/4}$ copies of either $F_0$ in $A$ or $Q$ in $B$. We take the vertices of each to form the set $U_1'$. Then we repeat this with copies of $F_0$ in $B$ and copies of $Q$ in $A$ to obtain $U_2'$, and then with every other bipartition of $F$ into two connected parts to obtain sets $U_i'$ of vertices. Let $U$ be the union of all the sets $U_j$ and $U_i'$.
 
 \begin{clm}
 $|U|=o(n)$.
 \end{clm}
 
 \begin{proof}
 For every $i$, $U_i'$ has cardinality at most $n^{3/4}$, and we have constant many of them, thus their total cardinality is $O(n^{3/4})$. We also have $O(1)$ copies of $U_j$, each of cardinality $O(|U_0|)$, thus it is enough to show that $|U_0|=o(n)$.
 
 Let $F'$ be the bipartite graph we obtain by deleting $w$ from $F$. By the K\H ov\'ari-T. S\'os-Tur\'an theorem \cite{kst}, $\ex(n,F')=o(n^2)$, thus there exists an $m$ such that if $n\ge m$, then $\ex(n,F')\le n^2/5$. Consider a vertex $v\in T$. For the vertices $v$ with $|U(v)|< m$, altogether at most $m|T|=o(n)$ vertices are in the sets $U(v)$. Let $T'$ denote the set of vertices $v\in T$ with $|U(v)|\ge m$. We take $|U(v)|$ vertices from both $A(v)$ and $B(v)$, and consider the bipartite graph $G(v)$ defined by the edges of $H$ between these subsets. Clearly $G(v)$ is $F'$-free, thus there are at most $4|U(v)|^2/5$ edges of $G'$ between these two parts and at least $|U(v)|^2/5$ edges are missing. This shows that $\Omega(|U(v)|^2)$ edges are missing between $A(v)$ and $B(v)$.
 
 We know that in total $o(n^{3/2})$ edges are missing between $A$ and $B$, thus $\sum_{v\in T'}|U(v)|^2=o(n^{3/2})$. On the other hand, by the Cauchy-Schwartz inequality we have that $\sum_{v\in T'}|U(v)|^2\ge (\sum_{v\in T'}|U(v)|)^2/|T'|$. This implies $\sum_{v\in T'}|U(v)|=o(n)$, finishing the proof.
 \end{proof}
 
 Let us return to the proof of the theorem. Let $G_1$ be the graph we obtain by deleting the vertices of $U$ from $H$. We will show that $G_1$ is a complete bipartite graph with partite sets $A'=A\setminus U$ and $B'=B\setminus U$. Assume indirectly that $u\in A'$, $v\in B'$ and $uv$ is not an edge of $G_1$, thus not an edge of $G$. Then adding the edge $uv$ to $G$ creates a copy of $F$, which we denote by $F^*$. The vertex $w$ of $F^*$ is either $u$, $v$, or connected to both $u$ and $v$, thus cannot be in $T$. Assume without loss of generality that $w\in A$. 

Let $R_1$ denote the set of the neighbors of $v$ in $F^*$, then they are either in $A$, or in $B$. But in $B$, elements of $R_1$ cannot belong to $U_i$ with $0\le i\le |V(F)|-1$ (as then $v$ would be in $U_{i+1}$ and not in $B'$). Let $R_j$ for $j<|V(F)|$ denote the set of neighbors of the vertices of $R_{j-1}\cap B$ in $F^*$. Then similarly, we have that elements of $R_j$ belong to $A$ or $B$, and those in $B$ cannot belong to $U_i$ with $0\le i\le |V(F)|-j$. $R_j$ stops increasing before we arrive to $U_0$, let $R$ denote the final $R_j$ obtained this way. Then 
%the vertices of $R\cap B$ are all in $A$. Also, 
$R\cap B$ induces a connected subgraph of $F^*$ by construction. Let $F_0$ denote this subgraph, $F_1$ denote the remaining part of $F^*$ and $Q$ denote the subgraph of $F^*$ induced on $R\cap A$. Then this is a partition as described in the construction of $U$, thus we have moved each vertices of $R\cap A$ or $R\cap B$ to $U$. In particular $u$ or $v$ is in $U$ and not in $G_1$, a contradiction.
\end{proof}

\begin{proof}[Proof of Theorem \ref{elkrit}] We start the proof similarly to that of Theorem \ref{main}, but the set $U$ of vertices we delete will be slightly different.
 We apply Corollary \ref{corola} with $\alpha=o(n^{-1/2})$ to obtain $H'$ with partite sets $A$ and $B$. Let $T$ denote the set of vertices not in $H'$, thus $|T|=o(n^{1/2})$. For each vertex $v\in T$, let $A(v)$ denote its neighborhood in $A$, $B(v)$ denote its neighborhood in $B$ and $U(v)$ denote the smaller of $A(v)$ and $B(v)$ (an arbitrary one of them in case they have the same size), as in the proof of Theorem \ref{main}. We let $U_0=\cup_{v\in T} U(v)$, thus $|U_0|=o(n)$ as in the proof of Theorem \ref{main}.
 
 Observe that every vertex $u\in A(v)$ is connected to less than $|V(F)|$ vertices in $B(v)$. Indeed, otherwise these $|V(F)|$ vertices have $|V(F)|$ other common neighbors in $A$ by the minimum degree condition, and these $2|V(F)|$ vertices together with $u$ and $v$ form $K_{|V(F)|,|V(F)|+2}$ with an additional edge in one of the parts. This subgraph clearly contains $F$, a contradiction.
 
 For every $v\in T$, this means that the vertices of $U(v)$ have at most $|V(F)||U(v)|$ common neighbors with $v$ in the other partite set of $H'$. Let $U'(v)$ be the set of those common neighbors and let $U'=\cup_{v\in T} U'(v)$. Then $|U'|\le |V(F)||U_0|=o(n)$.
 
 Consider now an edge $uv$ inside $T$. We have that $u$ and $v$ have less than $|V(F)|$ common neighbors in $A$ (and at most $|V(F)|$ common neighbors in $B$) by the same reasoning: otherwise we can find $|V(F)|$ common neighbors of those vertices in $B$ by the minimum degree condition, giving us a copy of $K_{|V(F)|,|V(F)|+2}$, a contradiction. Let $U''(uv)$ denote the set of these less than $2|V(F)|$ common neighbors and let $U''=\cup_{u,v\in T, uv\in E(G)} U''(uv)$. Clearly $|U''|=o(n)$.
 
 Let $U=U_0\cup U'\cup U''$.
 We delete $U$ from $H'$ to obtain $G_2$. 
 
 Assume that $u\in A\setminus U'$, $v\in B\setminus U'$ and $uv$ is not an edge in $G_2$, thus not an edge in $G$. Then adding the edge $uv$ to $G$ creates a copy of $F$ that we denote by $F^*$. There is a triangle containing $u$ or $v$, say $uxy$ in $F^*$. One of its vertices, say $x$ is in $T$, since $H'$ is bipartite. Then $U(x)=B(x)$, since $u\not\in U_0$. As $y$ is connected to $u$, we have $y\in B\cup T$. If $y\in B$, then $y\in U_0$, but then its common neighbors with $x$, including $u$, were moved to $U'$, a contradiction. Thus $y\in T$, but then $u\in U''(xy)\subset U$, a contradiction.
\end{proof}

\end{document}